\documentclass[11pt]{amsart}
\makeatletter
\@namedef{subjclassname@2020}{2020 Mathematics Subject Classification}
\makeatother

\usepackage{amsmath,amsfonts,amsthm,amssymb,graphicx,tikz,tikz-cd,url,hyperref,cleveref,mathtools}
\usepackage[all,2cell,ps]{xy}
\usepackage{fullpage}  

\theoremstyle{plain}
\newtheorem{thm}{Theorem}[section]

\newtheorem{prop}[thm]{Proposition}

\newtheorem{cor}[thm]{Corollary}

\newtheorem{qtns}[thm]{Questions}

\theoremstyle{definition}

\newtheorem{rem}[thm]{Remark}

\theoremstyle{remark}

\newcommand{\bbB}{\mathbb{B}}
\newcommand{\bbC}{\mathbb{C}}

\newcommand{\bbH}{\mathbb{H}}

\newcommand{\bbP}{\mathbb{P}}
\newcommand{\bbQ}{\mathbb{Q}}
\newcommand{\bbR}{\mathbb{R}}

\newcommand{\bbZ}{\mathbb{Z}}

\newcommand{\calA}{\mathcal{A}}

\newcommand{\calC}{\mathcal{C}}
\newcommand{\calD}{\mathcal{D}}
\newcommand{\calE}{\mathcal{E}}

\newcommand{\calH}{\mathcal{H}}

\newcommand{\calM}{\mathcal{M}}

\newcommand{\al}{\alpha}
\newcommand{\gam}{\gamma}
\newcommand{\Gam}{\Gamma}

\newcommand{\Del}{\Delta}
\newcommand{\ep}{\epsilon}

\newcommand{\Lam}{\Lambda}

\newcommand{\om}{\omega}

\DeclareMathOperator{\U}{U}

\DeclareMathOperator{\SL}{SL}

\DeclareMathOperator{\PGL}{PGL}

\DeclareMathOperator{\PU}{PU}
\DeclareMathOperator{\SU}{SU}

\DeclareMathOperator{\Sp}{Sp}

\DeclareMathOperator{\Id}{Id}

\DeclareMathOperator{\Aut}{Aut}

\DeclareMathOperator{\Out}{Out}

\newcommand{\bs}{\backslash}

\newcommand{\lra}{\longrightarrow}

\newcommand{\ssm}{\smallsetminus}

\DeclareMathOperator{\PGam}{\mathrm{P}\Gamma}
\DeclareMathOperator{\SR}{\mathcal{SR}}

\newcommand{\wh}{\widehat}
\newcommand{\wt}{\widetilde}

\newenvironment{pf}{\begin{proof}}{\end{proof}}

\allowdisplaybreaks

\makeatletter
\let\@@pmod\pmod
\DeclareRobustCommand{\pmod}{\@ifstar\@pmods\@@pmod}
\def\@pmods#1{\mkern4mu({\operator@font mod}\mkern 6mu#1)}
\makeatother

\usetikzlibrary{arrows}

\title[Automorphisms of the moduli of cubic surfaces]{Automorphisms of the 
moduli space of smooth cubic surfaces and its fundamental group}

\author[Baldi]{Gregorio Baldi}
\address{Gregorio Baldi\\ CNRS, IMJ-PRG, Sorbonne Universit\'{e}, 4 place Jussieu, 75005 Paris, France }
\email{ baldi@imj-prg.fr }

\author[Farb]{Benson Farb}
\address{Benson Farb \\ Dept. of Mathematics, University of Chicago, 5734 University Avenue Chicago, IL 60637-1514, USA}
\email{bensonfarb@gmail.com}

\author[Javanpeykar]{Ariyan Javanpeykar}
\address{Ariyan Javanpeykar \\
IMAPP Radboud University Nijmegen,
PO Box 9010, 6500GL,
Nijmegen, The Netherlands}
\email{ariyan.javanpeykar@ru.nl}

\author[Stover]{Matthew Stover}
\address{Matthew Stover, Dept. of Mathematics, Temple University, Philadelphia, PA 19122, USA}
\email{mstover@temple.edu}



\date{\today}

\begin{document}

\maketitle

\begin{abstract}
Let $\calC$ be the moduli space of smooth complex cubic surfaces and let $\pi_1(\calC)$ be its (orbifold) fundamental group. We prove that the ``divisor subgroup'' of $\pi_1(\calC)$ is characteristic. This can be interpreted as saying that the group theory of $\pi_1(\calC)$ ``remembers'' the divisor of nodal cubic surfaces. We deduce from this group-theoretic result and some basic complex analysis that $\calC$ has no nontrivial biholomorphic automorphisms as complex analytic orbifold.
\end{abstract}

\section{Introduction}\label{sec:Intro}

Let $\calC$ be the moduli space of smooth, complex cubic surfaces, regarded as a complex analytic orbifold. In this paper we prove two main results. First, we prove that ``group theory remembers algebraic geometry'': every automorphism of the orbifold fundamental group $\pi_1(\calC)$ leaves invariant the ``divisor subgroup'' corresponding to the divisor $D$ of nodal cubics surfaces; see below for a precise statement. We then apply this group theory result to prove the following geometric result.

\begin{thm}\label{theorem:main1}
The group $\Aut(\calC)$ of biholomorphic automorphisms of the moduli space $\calC$ of smooth complex cubic surfaces is the trivial group.
\end{thm}

Here automorphisms of $\calC$ are with respect to the analytic orbifold structure on $\calC$ as described in \Cref{sec:GoodOrb}. \Cref{theorem:main1} complements the classical theorems
	\begin{align*}
	\Aut(\calM_g) = 1 && \Aut(\calA_g) = 1
	\end{align*}
for the moduli spaces $\calM_g$ of smooth curves with genus $g$ \cite{zbMATH03352015} and $\calA_g$ of $g$-dimensional principally polarized abelian varieties. More broadly, these results all fit within a larger program studying rigidity theorems for moduli spaces. See \cite{zbMATH07925348} for further context, references, and discussion.

Underpinning rigidity of $\calA_g$ is the fact that it is a \emph{locally symmetric variety}, meaning that it is the quotient of a hermitian symmetric domain by a discrete group $\Gam$ of biholomorphic automorphisms\footnote{This definition includes the case where $\Gam$ has torsion and the quotient is only a locally symmetric orbifold, and hence may only be a normal variety.}. More specifically, $\calA_g$ is the quotient of the Siegel upper half-plane $\bbH_g$ by $\Gam \coloneqq \Sp(2g,\bbZ)$. Then $\Aut(\Gam \backslash \bbH_g)$ is naturally identified with the normalizer of $\Gam$ in the real semisimple Lie group $\Aut(\bbH_g)$, reducing the computation of $\Aut(\Gam \backslash \bbH_g)$ to a problem about the normalizer in $\Sp(2g, \bbR)$ of the arithmetic group $\Gamma$.

This strategy does not work for the moduli space $\calC$. Indeed, $\calC$ is not even homotopy equivalent to a locally symmetric variety (see \cite[Thm.\ 1.2]{ACTorthogonal}). However, as explained in more detail in \Cref{sec:Prelims} below, Allcock--Carlson--Toledo \cite{ACT} proved using Hodge theory that there is a locally symmetric variety
	\[
	X \coloneqq \PU(4,1)(\calE) \bs \bbB^4
	\]
with $\calE \coloneqq \bbZ[e^{2\pi i/3}]$, a totally geodesic, immersed divisor
	\[
	D \cong \PU(3,1)(\calE) \bs \bbB^3
	\]
in $X$, and an isomorphism
	\[
		\calC\stackrel{\cong}{\lra} (X \ssm D)
	\]
of complex orbifolds. In \cite{ACT} they show that the divisor $D$ has a modular interpretation: it is isomorphic to the moduli space of singular cubic surfaces with only nodal singularities.

If an automorphism of $\calC\cong (X \ssm D)$ extended to an automorphism of the locally symmetric variety $X$, we could apply the strategy above to prove \Cref{theorem:main1}. The problem is that such a phenomenon is false in general: there are even examples of locally symmetric varieties $Y_i$ with divisors $D_i$ so that $(Y_1 \ssm D_1)\cong (Y_2 \ssm D_2)$ but $Y_1$ is not isomorphic to $Y_2$. See \cite{zbMATH00606995} and \cite{MR1136204} for examples.

Therefore, the main issue in our proof of \Cref{theorem:main1} is that automorphisms of $X \ssm D$ ``remember the divisor $D$''. Since $X$ is a ball quotient, it turns out that this is essentially a group-theoretic issue, as follows. The inclusion $i:\calC\cong (X \ssm D)\to X$ induces a surjection $i_*:\pi_1(\calC)\to\pi_1(X)$ of (orbifold) fundamental groups (see \Cref{sec:GoodOrb}) with kernel $K$, giving an exact sequence:
	\begin{equation}\label{eq:exact1}
		1 \lra K \lra \pi_1(\calC) \lra \PU(4,1)(\calE)\lra 1
	\end{equation}
Here $K=\pi_1(\bbB^4 \ssm \calH)$ where $\calH$ is the union of the set of all lifts of the divisor $D$ to the universal cover $\widetilde{X}\cong\bbB^4$. The key group theory result we use to prove \Cref{theorem:main1} is the following, which can be interpreted as ``the group $\pi_1(\calC)$ 
remembers the divisor $D$''.

\begin{thm}\label{thmKchar}
The kernel $K$ of the projection $\pi_1(\calC) \to \PU(4,1)(\calE)$ in \Cref{eq:exact1} is a \emph{characteristic} subgroup of $\pi_1(\calC)$; that is, $\phi(K)=K$ for every $\phi\in\Aut(\pi_1(\calC))$.
\end{thm}

\begin{rem}
\Cref{thmKchar} is not true for general reasons. As a simple example, let $S$ be a genus $2$ curve. Fix a point $p\in S$ and consider the divisor $D \coloneqq \{p\}$. One can show that the kernel $K$ of the natural surjection $\pi_1(S\ssm D) \to \pi_1(S)$ is not a characteristic subgroup of ${\pi_1(S\ssm D)}$.
\end{rem}

\begin{rem}[Alternate approach]\label{rmk0} 
As pointed out to us by Curt McMullen, an alternative proof of \Cref{theorem:main1} can be given using the \emph{Eckardt divisor}. This is the divisor parametrizing smooth cubic surfaces admitting an Eckardt point, or equivalently a complex reflection of order 2. In the ball quotient model of Allcock--Carlson--Toledo \cite{ACT}, it is the image of the union of the hyperplanes orthogonal to the long roots of the relevant Eisenstein lattice. See \Cref{rmk1} for a more detailed discussion.
\end{rem}

\noindent
\textbf{Further questions.} It would be interesting to understand automorphisms of, and more generally morphisms between, divisor complements of locally symmetric varieties. Several moduli spaces are of this form, such as the moduli space of smooth quartic curves. We have been able to use the methods of this paper together with Margulis superrigidity to classify automorphisms of $X \ssm D$ where $X$ (resp.\ $D$) is locally symmetric of type $\textrm{IV}_n$ (resp.\ $\textrm{IV}_{n-1}$). However, in the ball quotient case of this paper, the special nature of self-intersections of $D$ is crucial.
\medskip

The use of the purely group-theoretic \Cref{thmKchar} in the proof of \Cref{theorem:main1} suggests the following questions, which also echo analogous results for arithmetic groups and mapping class groups.

\begin{qtns}{\ }
\begin{enumerate}

\item Is $\Out(\pi_1(\calC)) \cong \bbZ / 2\bbZ$, generated by the outer automorphism induced by complex conjugation?

\item Is the abstract commensurator of $\pi_1(\calC)$ isomorphic to $\pi_1(\calC) \rtimes \langle \ep \rangle$, where $\ep$ is the automorphism induced by complex conjugation?

\end{enumerate}
\end{qtns}

\subsubsection*{Organization of the paper}

Preliminary results are contained in \Cref{sec:Prelims}, and some preliminary results needed to prove the main theorems are contained in \Cref{sec:More}. The proof of \Cref{thmKchar} is in \Cref{sec:CharProof}, and \Cref{theorem:main1} is proved in 
\Cref{sec:MainProof}.

\subsubsection*{Acknowledgments}

The authors thank Curtis McMullen and Daniel Allcock for comments on an early draft of the paper. A.J. was supported by the Netherlands Organization for Scientific Research NWO, through grant OCENW.M.23.206. B.F.\ is supported by National Science Foundation grant DMS-2203355. During the final stages of this work, G.B.\ was at the Institute for Advanced Study in Princeton and he is grateful for the amazing working conditions. He also thanks the Ambrose Monell Foundation and the Giorgio and Elena Petronio Fellowship Fundas well as the Agence Nationale de la Recherche for support (grant ANR-HoLoDiRibey). M.S.\ was supported by NSF grants DMS-2203555 and DMS-2506896, along with award SFI-MPS-TSM-00014184 from the Simons Foundation. This material is based upon work supported by the National Science Foundation under Grant No. DMS-1928930, while the author was in residence at the Simons Laufer Mathematical Sciences Institute in Berkeley, California, during the Spring semester 2026.

\section{$\calC$ as a divisor complement in a Picard modular $4$-fold}\label{sec:Prelims}

The goal of this section is to explain how $\calC$ is a divisor complement in a certain Picard modular $4$-fold, a locally symmetric orbifold.

\subsection{Good orbifolds}\label{sec:GoodOrb}

All complex analytic orbifolds in this paper will be \emph{good orbifolds}: any 
such $X$ can be realized as a quotient
	\[
	X = \Gam \bs \wt{X}
	\]
of a simply-connected complex manifold $\wt{X}$ by a group $\Gam$ acting effectively and properly discontinuously on $\wt{X}$ by biholomorphic automorphisms. Then
	\[
	\pi_1(X) \coloneqq \Gam
	\]
is the \emph{orbifold fundamental group} of $X$. The usual homotopy lifting properties hold in the category of good orbifolds, so a \emph{biholomorphic automorphism} of $X$ is equivalent to a $\Gam$-equivariant biholomorphic automorphism of $\wt{X}$. The group of biholomorphic automorphisms of a good orbifold $X$ is denoted $\Aut(X)$.

\subsection{Complex hyperbolic space}\label{backgroundcomplexhyp}

See \cite{zbMATH01270600} for a general reference on complex hyperbolic geometry and \cite[Ch.\ II.10]{zbMATH01385418} for a more general discussion of the geometry of rank one symmetric spaces. For each $n \ge 1$, the negative lines in $\bbP^n$ for the hermitian form
	\begin{equation}\label{eq:HermitianForm}
	h_n(z) \coloneqq -|z_0|^2 + \cdots + |z_n|^2
	\end{equation}
on $\bbC^{n+1}$ define a projective embedding
	\[
	\bbB^n \coloneqq \!\left\{[z] \in \bbP^n\ :\ h_n|_{[z]} < 0\right\}\! \subset\bbP^n
	\]
of the $n$-dimensional complex unit ball. The hermitian metric induced by $h_n$ is the Bergman metric on $\bbB^n$, which has biholomorphic automorphism group
	\[
	\PU(n,1)\coloneqq\!\left\{g \in \PGL_{n+1}(\bbC)\ :\ g(\bbB^n) = \bbB^n \right\}\!.
	\]
Equipped with this metric, $\bbB^n$ is \emph{complex hyperbolic $n$-space}, and it is naturally $\PU(n,1) / \U(n)$ under the left action of $\PGL_{n+1}$ on $\mathbb{P}^n$. Holomorphic totally geodesic embeddings $\bbB^{n-1}\hookrightarrow \bbB^n$ are in one-to-one correspondence with $h_n$-positive lines $\ell$ in $\bbP^n$, where the embedding is explicitly given by the orthogonal complement $\ell^\perp$. Moreover, the \emph{boundary at infinity} $\partial \bbB^n$ is identified with the space of $h_n$-isotropic lines in $\bbC^{n+1}$ and is homeomorphic to the $(2n-1)$-sphere.

The classification of holomorphic isometries of $\bbB^n$ is as follows. An element of $\PU(n,1)$ is called:
\begin{itemize}

  \item \emph{elliptic} if it has a fixed point in $\bbB^n$;

  \item \emph{parabolic} if it has no fixed point in $\bbB^n$ and exactly one fixed point on $\partial\bbB^n$; and

  \item \emph{loxodromic} if it has no fixed point in $\bbB^n$ and exactly two fixed points on $\partial\bbB^n$.

\end{itemize}
A \emph{complex reflection} is an elliptic element whose fixed-point set in $\bbB^n$ has complex codimension $1$. The fixed set of a complex reflection is then a totally geodesic embedding of $\bbB^{n-1}$.

\subsection{Eisenstein--Picard modular $n$-folds}\label{hypprel}

 Let $\omega$ be a primitive $3^{rd}$ root of unity and let $\calE = \bbZ[\omega]$ be the Eisenstein integers. For each $n \ge 1$,
	\[
	\Gam_n \coloneqq \SU(n,1)(\calE)
	\]
will denote the subgroup of $\SL_{n+1}(\calE)$ preserving the form $h_n$ from \Cref{eq:HermitianForm}. The image of $\Gam_n$ in $\PU(n,1)$ under projection will be denoted $\PGam_n$. Then $\PGam_n$ acts effectively and properly discontinuously on $\bbB^n$ by biholomorphic automorphisms, but not freely since it has nontrivial torsion.

The \emph{Picard modular orbifold} 
	\[
		X_n \coloneqq \PGam_n \bs \bbB^n
	\]
is a complete, complex-hyperbolic orbifold, and is noncompact for all $n \ge 2$. Moreover, Baily--Borel \cite{BailyBorel} proved that $X_n$ is a normal quasiprojective variety. If $m < n$, the standard inclusion $\bbB^m \hookrightarrow \bbB^n$ is equivariant for the analogous inclusion $\Gam_m \le \Gam_n$, inducing a totally geodesic immersion $X_m \looparrowright X_n$ of complex hyperbolic orbifolds.

\medskip

\noindent\textbf{Notation.} Throughout this paper, $X \coloneqq X_4$ and $D$ will denote the image of $X_3$ in $X$.

\begin{rem}\label{rem:ProjectiveEmbed}
Note that $\PGam_m$ is not a subgroup of $\PGam_n$. Indeed, the projection of $\SU(n,1)$ to $\PU(n,1)$ is injective on $\SU(m,1)$ for $m < n$, and the stabilizer in $\PU(n,1)$ of $\bbB^m$ is isomorphic to
	\[
	\mathrm{P}\!\left(\U(m,1) \times \U(n-m)\right)
	\]
with $\U(n-m)$ acting effectively on the normal bundle to $\bbB^m$ in $\bbB^n$.
\end{rem}

\subsection{Automorphisms of $X$}\label{ssec:X4auts}

In this subsection we recall some facts regarding the automorphisms of $\PGam_4$ and $X$. The following is known to experts.

\begin{prop}\label{prop:Maximal}
The lattice $\PGam_4 < \PU(4,1)$ is maximal with respect to inclusion of lattices.
\end{prop}
\begin{pf}
Note that the preimage $\wt{\Gam}_4$ of $\PGam_4$ in $\SU(4,1)$ is a central extension of $\PGam_4$ by the center $\bbZ / 5\bbZ$ of $\SU(4,1)$. In particular, $[\wt{\Gam}_4 : \Gam_4] = 5$. It suffices to show that $\wt{\Gam}_4$ is a maximal lattice in $\SU(4,1)$. To prove this, \cite[Prop.\ 3.8]{MR3163356} shows that the normalizer $\mathrm{N}(\wt{\Gam}_4)$ of $\wt{\Gam}_4$ in $\SU(4,1)$ is a maximal lattice. Since $\bbQ(\om)$ has class number one, \cite[Lem.\ 3.7]{MR3163356} implies that $[\mathrm{N}(\wt{\Gam}_4): \Gam_4] = 5$, and the result follows.
\end{pf}

Mostow--Prasad rigidity (see \cite{Prasad}) gives the following immediate corollary.

\begin{cor}\label{cor:ArithNoAuts}
The orbifold $X$ admits no nontrivial holomorphic automorphisms.
\end{cor}

\begin{rem}\label{rem:Anti1}
Note that complex conjugation on $\bbB^4$, which is compatible with complex conjugation on $\PGL_5(\bbC)$ and stabilizes the standard embedding of $\bbB^5$, induces an antiholomorphic involution of $X$ preserving $D$.
\end{rem}

\subsection{The period mapping}\label{periodmapsec}

Allcock--Carlson--Toledo \cite{ACT} constructed a period mapping 
	\[
		\calC\to X \ssm D
	\]
where $X$ and $D$ are as in \Cref{hypprel}. To describe the image of $\calC$ in $X$ more precisely, let $S$ be a smooth cubic surface in $\bbP^3$ and let $T$ be the triple cover of $\bbP^3$ branched along $S$. The natural $\bbZ / 3\bbZ$ action on $T$ induces an action of $\omega$ on $H^3(T;\bbZ)$ making it a free $\calE$-module of rank $5$, and the cup product defines a unimodular $\calE$-hermitian form on $H^3(T;\bbZ)$ of signature $(4,1)$.

The map assigning to $S$ the polarized Hodge structure on $H^3(T;\bbZ)$ defines a period map valued in $\bbB^4$ with monodromy group a discrete subgroup of $\mathrm{U}(4,1)$ generated by complex reflections of order six. The projectivization of the monodromy group is isomorphic to $\PGam_4$. Moreover, there is a surjection 
	\[
		\PGam_4\lra \mathrm{W}(\mathrm{E}_6)
	\]
onto the Weyl group $\mathrm{W}(\mathrm{E}_6)$, the automorphism group of the $27$ lines on a cubic surface.

Let $\Lam \coloneqq \calE^{4,1}$ be the free module $\calE^5$ equipped with the Hermitian form $h_4$ defined by \Cref{eq:HermitianForm}. Then $\Lam$ is isometric to the $\calE$-module $H^3(T;\bbZ)$. A vector in $\Lam$ with norm $1$ is called a \emph{short root}. Let $\SR$ denote the collection of short roots. Since totally geodesic embeddings of $\bbB^3$ into $\bbB^4$ are in one-to-one correspondence with $h_4$-positive lines, each short root $r$ determines a totally geodesic inclusion $\bbB^3_r\lhook\joinrel\longrightarrow \bbB^4$. The short roots thus define the \emph{hyperplane arrangement}
	\[
		\calH \coloneqq \bigcup_{r\in \SR}\bbB^3_r
	\]
that is clearly preserved by $\PGam_4$. The following now summarizes several of the fundamental results from \cite{ACT}.

\begin{thm}[Allcock--Carlson--Toledo \cite{ACT}]\label{thm:ACTbig}
With the notation established in this section:
\begin{enumerate}

	\item The image of the period map $\calC \to \PGam_4 \bs \bbB^4$ is precisely $X \ssm D$.

	\item If $r_1, r_2 \in \SR$ then $\bbB^3_{r_1}$ and $\bbB^3_{r_2}$ are either disjoint or meet orthogonally.

	\item The action of $\PGam_4$ on the set of lines spanned by short roots is transitive.

\end{enumerate}
\end{thm}

To be precise, the image of the period mapping is $\Gam_4 \bs\!\left(\bbB^4 \ssm \calH\right)$, and \Cref{thm:ACTbig}(3) then identifies the image with $X \ssm D$, since the standard embedding of $\bbB^3$ is associated with a short root. In particular, $D$ is the image of $\calH$ in $X$.

\subsection{Generators for $\pi_1(\calC)$}\label{ssec:Gens}

In the process of proving the results summarized in \Cref{thm:ACTbig}, a collection of seven generators for $\PGam_4$ is studied in \cite[\S 7]{ACT}, where the reader should be warned that an integral change of basis has been performed so that the hermitian form has the matrix $A$ given by \cite[Eq.\ (7.7.1)]{ACT}. Specifically, these generators are \emph{hexaflections}, meaning complex reflections of order $6$ fixing $\bbB^3_r$ for some root $r$ and acting by $-\omega$ on the normal bundle to $\bbB^3_r$. These generators are the image in $\PGam_4$ of Looijenga's generators for $\pi_1(\calC)$ realizing $\pi_1(\calC)$ as a quotient of the Artin group $\mathrm{A}(\widetilde{\mathrm{E}}_6)$ on the affine $\mathrm{E}_6$ diagram given in \Cref{fig:ArtinGens} \cite{zbMATH05254786}. The first six reflection generators are Libgober's \cite{zbMATH03574026}. In \Cref{fig:ArtinGens}, generators of $\mathrm{A}(\widetilde{\mathrm{E}}_6)$ mapping to hexaflections through a collection of mutually intersecting hyperplanes associated with short roots are circled. Note that there is a typo in \cite[Eq.\ (7.7.2)]{ACT}, where $r_7$ should be $(0,0,1,0,0)$, as defined in \cite[Lem.\ 7.17]{ACT}. The following compiles these results.

\begin{thm}\label{thm:ACTgensE6}
Let $R_j \in \PGam_4$ be the images in $\PU(4,1)$ of hexaflections through the roots $r_j$ given in \cite[\S 7]{ACT}. Then the $R_j$ satisfy the relations of $\mathrm{A}(\widetilde{\mathrm{E}}_6)$ on the affine $\mathrm{E}_6$ diagram, where the circled vertices in \Cref{fig:ArtinGens} indicate hexaflections through mutually intersecting hyperplanes associated with short roots. In particular, there is a surjective homomorphism $\mathrm{A}(\widetilde{\mathrm{E}}_6)\to \PGam_4$ factoring through the representation $\pi_1(\calC) \to \PGam_4$.
\end{thm}

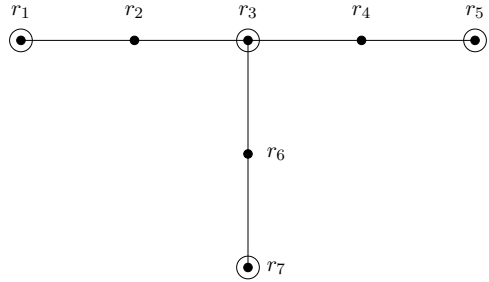
\begin{figure}[h]
\centering
\scalebox{0.75}{
\begin{tikzpicture}
\draw[fill=black] (0,0) circle (0.075cm) node [yshift=0.5cm] {$r_3$};
\draw (0,0) circle (0.2cm);
\draw[fill=black] (2,0) circle (0.075cm) node [yshift=0.5cm] {$r_4$};
\draw[fill=black] (-2,0) circle (0.075cm) node [yshift=0.5cm] {$r_2$};
\draw[fill=black] (4,0) circle (0.075cm) node [yshift=0.5cm] {$r_5$};
\draw (4,0) circle (0.2cm);
\draw[fill=black] (-4,0) circle (0.075cm) node [yshift=0.5cm] {$r_1$};
\draw (-4,0) circle (0.2cm);
\draw[fill=black] (0,-2) circle (0.075cm) node [xshift=0.5cm] {$r_6$};
\draw[fill=black] (0,-4) circle (0.075cm) node [xshift=0.5cm] {$r_7$};
\draw (0,-4) circle (0.2cm);
\draw (-4,0) -- (4,0);
\draw (0,0) -- (0,-4);
\end{tikzpicture}
}
\caption{The hexaflection generators for $\pi_1(\calC)$ from $\mathrm{A}(\widetilde{\mathrm{E}}_6)$.}\label{fig:ArtinGens}
\end{figure}

\section{Some properties of the arrangement $\calH$}\label{sec:More}

In this section we prove two basic results about the hyperplane arrangement $\calH$ that will be used in proving the main theorems of this paper.

\subsection{The hyperplane arrangement is connected}\label{ssec:Connected}

The following consequence of \Cref{thm:ACTgensE6}, which was perhaps known to Allcock--Carlson--Toledo, plays a significant role in this paper.

\begin{prop}\label{prop:ConnectedSupport}
The arrangement $\calH$ in $\bbB^4$ defined in \Cref{sec:Prelims} has connected support.
\end{prop}
\begin{pf}
Consider the hyperplanes $\calD_k = \bbB^4_{r_{2 k - 1}}$ with $r_{2 k - 1}$ as in \Cref{thm:ACTgensE6}, $k = 1,2,3$, and set
	\[
	\calH_0 \coloneqq \bigcup_{k = 1,2,3} \calD_k.
	\]
\Cref{thm:ACTbig} implies that the hyperplanes in the support of $\calH_0$ meet orthogonally, since the hexaflections $R_j$ all commute. Moreover, each generator $R_j$ of $\PGam_4$, $1 \le j \le 7$, has the property that $R_j(\calH_0) \cap \calH_0$ is nontrivial. Indeed, each $\bbB^4_{r_j}$ meets some $\calD_k$ nontrivially, hence $R_j$ acts trivially on some totally geodesic subspace of $\calH_0$.

Since the $R_j$ generate $\PGam_4$ and $\PGam_4$ acts transitively on the set of short roots, it follows by induction that $\calH$ has connected support. Indeed, if
	\[
	\calH_d \coloneqq \bigcup_{|\gam| \le d}\gam(\calH_0)
	\]
for $|\gam|$ the word length in the generators $R_j$, then $\calH_0$ is connected for the base case. If $\calH_d$ is connected and $R_{j_1} \cdots R_{j_{d+1}}$ is a word of length $d+1$ in the generators $R_j$, then $R_{j_{d+1}}(\calH_0) \cap \calH_0 \neq \emptyset$, so
	\[
	(R_{j_1} \cdots R_{j_{d+1}})(\calH_0) \supseteq (R_{j_1} \cdots R_{j_d})\!\left(R_{j_{d+1}}(\calH_0) \cap \calH_0\right)\! \neq \emptyset
	\]
has nontrivial intersection with $(R_{j_1} \cdots R_{j_d})(\calH_0) \subseteq \calH_d$, which completes the induction.
\end{pf}

\subsection{Meridians and the monodromy exact sequence}\label{ssec:Exact}

The orbifold isomorphism between $\calC$ and $X \ssm D$ makes the analytic manifold $\bbB^4 \ssm \calH$ into an orbifold cover of $\calC$ with deck group $\PGam_4$. See \cite[Thm.\ 2.20]{ACT}. As described in \Cref{ssec:Connected}, the image in $\PU(4,1)$ of the monodromy representation takes Looijenga's generators for $\pi_1(\calC)$ \cite{zbMATH05254786} (which contain Libgober's \cite{zbMATH03574026}) to the hexaflections $R_j$ in short roots from \Cref{thm:ACTgensE6}, and these generate $\PGam_4$. In fact, $\PGam_4$ is normally generated by a single hexaflection in a short root by \Cref{thm:ACTbig}(iii) and \cite[Thm.\ 2.14]{ACT}. The kernel of the monodromy is then the fundamental group of $\bbB^4 \ssm \calH$.

Looijenga's generators for $\pi_1(\calC)$ are \emph{meridians}, i.e., loops around $D$. To be precise, consider the cover of $\calC$ by $\bbB^4 \ssm \calH$ and a short root $r$. The normal bundle to $\bbB^3_r$ in $\bbB^4$ is trivial, so choose a small holomorphic disk $\wt{\Del}_r$ in a fiber. Let $\Del_r$ be the image of $\wt{\Del}_r$ in $X$. Projection onto $X$ maps $\bbB^3_r$ to $D$ under the conjugated action of $\Gam_3$, and the action of the hexaflection in $r$ on the normal bundle to $\bbB^3_r$ is $w \mapsto w^6$ in suitable coordinates for $\wt{\Del}_r \to \Del_r$. See \Cref{fig:Meridian}. The following result concerning generators is implicit in \cite{ACT} but not stated there in this manner.
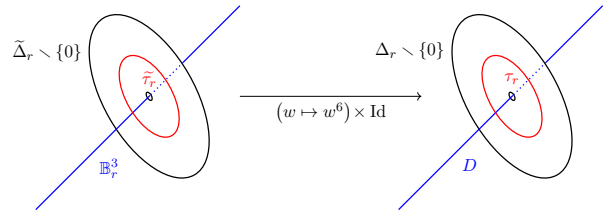
\begin{figure}[h]
\centering
\scalebox{0.6}{
\begin{tikzpicture}
\draw[rotate=30, thick] (0,0) circle [x radius = 0.05cm, y radius = 0.1cm];
\draw[rotate=30, thick] (0,0) circle [x radius = 1cm, y radius = 2cm];
\draw[thick, blue] (-2.5,-2.5) -- node[below right] {$\bbB^3_r$} (0.02,0.02);
\draw[rotate=30, red, thick] (0,0) circle [x radius = 0.5cm, y radius = 1cm] node[yshift = 1em] {$\wt{\tau_r}$};
\draw[thick, blue, dotted] (0.75,0.75) -- (0,0);
\draw[thick, blue] (0.75,0.75) -- (2,2);
\node at (-2.25,1) {$\wt{\Del}_r \ssm \{0\}$};
\draw[->] (2,0) -- node[below] {$\left(w \mapsto w^6\right)\! \times \Id$} (6,0);
\draw[rotate around={30:(8,0)}, thick] (8,0) circle [x radius = 0.05cm, y radius = 0.1cm];
\draw[rotate around={30:(8,0)}, thick] (8,0) circle [x radius = 1cm, y radius = 2cm];
\draw[thick, blue] (5.5,-2.5) -- node[below right] {$D$} (8.02,0.02);
\draw[rotate around={30:(8,0)}, red, thick] (8,0) circle [x radius = 0.5cm, y radius = 1cm] node[yshift = 1em] {$\tau_r$};
\draw[thick, blue, dotted] (8.75,0.75) -- (8,0);
\draw[thick, blue] (8.75,0.75) -- (10,2);
\node at (5.75,1) {$\Del_r \ssm \{0\}$};
\end{tikzpicture}
}
\caption{The meridian around $D$ locally under the orbifold covering.}\label{fig:Meridian}
\end{figure}

\begin{prop}\label{prop:exact_sequence}
There is a short exact sequence
	\begin{equation}\label{eq:exact_sequence}
	1 \lra K \lra \pi_1(\calC) \lra \PGam_4 \lra 1
	\end{equation}
with $K\cong\pi_1(\bbB^4 \ssm H)$. Moreover, $\pi_1(\calC)$ is normally generated by any single meridian $\tau$ around $D$ and $K$ is the normal closure in $\pi_1(\calC)$ of $\tau^6$.
\end{prop}
\begin{pf}
The first statement is a direct consequence of the surjectivity of $\pi_1(\calC) \to \PGam_4$ and orbifold covering space theory. Continuing with the notation of this subsection, the elements $\wt{\tau}_r$ representing positively oriented loops in $\Del_r \ssm \{0\}$ around $\bbB^3_r$ in $\bbB^4 \ssm \calH$ for $r \in \SR$ certainly generate $\pi_1(\bbB^4 \ssm \calH)$, since $\bbB^3$ and $\bbB^4$ are both contractible; see \cite{ACTorthogonal}. The image $\tau_r$ of $\wt{\tau}_r$ in $\pi_1(\calC)$ is then the sixth power of a meridian around $D$.

The covering to $X \ssm D \cong \calC$ then realizes $\wt{\tau}_r$ as $\tau_r^6$ in $\pi_1(\calC)$ by the orbifold path lifting property. Similarly, the inclusion $X \ssm D\hookrightarrow X$ induces the surjection of $\pi_1(\calC)$ onto $\PGam_4$ simply by imposing the relation $\tau_r^6 = 1$ (cf.\ \cite[\S 3]{zbMATH05254786}). The fact that $\PGam_4$ acts transitively on the hyperplanes in $\calH$ implies that for all $r, s \in \SR$, there is an element of $\PGam_4$ conjugating $\tau_r$ to $\tau_s$. Since $\PGam_4$ is generated by hexaflections through short roots, it is thus normally generated by a single one of them. The preimage in $\pi_1(\calC)$ of any such conjugating element conjugates $\wt{\tau}_r$ to $\wt{\tau}_s$, so the same conclusion holds in $K$. The proposition follows.
\end{pf}

\section{The proof of \Cref{thmKchar}}\label{sec:CharProof}

Continuing with the notation from \Cref{sec:Prelims}, let $K$ be the kernel of the monodromy action of $\pi_1(\calC)$ on $\bbB^4$ as in \Cref{eq:exact1} and
	\[
		\rho : \pi_1(\calC) \to \PGam_4
	\]
be that surjection. To prove that $K$ is characteristic in $\pi_1(\calC)$, it suffices to show that $\phi(K)\leq K$ for all $\phi \in \Aut(\pi_1(\calC))$. Indeed, if $\phi(K) \le K$ for all $\phi$ then $\phi^{-1}(K)\leq K$ as well, so
	\[
	K=(\phi\circ\phi^{-1})(K)\leq\phi(K)\leq K
	\]
implies that $K=\phi(K)$. Fix $\phi \in \Aut(\pi_1(\calC))$. Recall that $\rho(\tau_r)$ is a hexaflection, so it is an elliptic element of $\PGam_4$ of order six.

We prove the theorem in steps.

\medskip
\noindent\textbf{Step 1.} Modulo conjugation in $\mathrm{W}(\mathrm{E}_6)$, there is a unique homomorphism $\pi_1(\calC)\to\mathrm{W}(\mathrm{E}_6)$.
\smallskip

\noindent
\textbf{Proof.} We prove this result with the aid of Magma \cite{MR1484478}. This was also known to Peter Huxford (private communication) by the same method.

\Cref{thm:ACTgensE6} implies that there is a canonical composition of surjections
\begin{equation}\label{eqarting}
\mathrm{A}(\widetilde{\mathrm{E}}_6) \lra \pi_1(\calC)\lra \PGam_4 \lra \mathrm{W}(\mathrm{E}_6)
\end{equation}
where the homomorphisms are, from left to right, given by taking Looijenga's generators, the monodromy action on $\bbB^4$, and finally the homomorphism to the symmetric group of the $27$ lines on a cubic surface. Critically for the proof of Step 2 below, by \cite[\S 3.12]{ACT} the homomorphism to $\mathrm{W}(\mathrm{E}_6)$ can also be described as reduction of the arithmetic group $\PGam_4$ modulo $(1-\om)$.

Using the sequence of homomorphisms in \Cref{eqarting}, it suffices to show that there is a unique homomorphism from $\mathrm{A}(\widetilde{\mathrm{E}}_6)$ onto $\mathrm{W}(\mathrm{E}_6)$. This is easily checked using the computer algebra program Magma using the \texttt{Homomorphisms} command. Specifically,

\medskip

{\scriptsize
\texttt{G := WeylGroup(RootSystem("E6"));}

\texttt{P := PermutationGroup(FPGroup(G));}

\texttt{A<g1,g2,g3,g4,g5,g6,g7> := Group<g1,g2,g3,g4,g5,g6,g7 |}

\texttt{\quad\quad (g1*g2*g1)=(g2*g1*g2), (g1*g3)=(g3*g1), (g1*g4)=(g4*g1), (g1*g5)=(g5*g1),}

\texttt{\quad\quad (g1*g6)=(g6*g1), (g1*g7)=(g7*g1), (g2*g3*g2)=(g3*g2*g3), (g2*g4)=(g4*g2),} 

\texttt{\quad\quad (g2*g5)=(g5*g2), (g2*g6)=(g6*g2), (g2*g7)=(g7*g2), (g3*g4*g3)=(g4*g3*g4),}

\texttt{\quad\quad (g3*g5)=(g5*g3), (g3*g6*g3)=(g6*g3*g6), (g3*g7)=(g7*g3),}

\texttt{\quad\quad (g4*g5*g4)=(g5*g4*g5), (g4*g6)=(g6*g4), (g4*g7)=(g7*g4), (g5*g6)=(g6*g5),} 

\texttt{\quad\quad (g5*g7)=(g7*g5), (g6*g7*g6)=(g7*g6*g7)>;}

\texttt{B, b := Simplify(A);}

\texttt{\#Homomorphisms(B, P);}
}

\medskip

\noindent quickly reports that there is exactly one homomorphism. More precisely, this means that up to automorphisms of $\mathrm{W}(\mathrm{E}_6)$ (which are all inner), there is a unique surjective homomorphism from $\mathrm{A}(\widetilde{\mathrm{E}}_6)$ onto $\mathrm{W}(\mathrm{E}_6)$, as claimed.

\medskip
\noindent\textbf{Step 2.} $t_\phi \coloneqq \rho(\phi(\tau_r))$ is an elliptic element of $\PGam_4$.

\medskip
\noindent
\textbf{Proof.} It suffices to show that $t_\phi$ is not a loxodromic or parabolic isometry. (It will be convenient for this proof to consider the trivial isometry as elliptic.) Let $\Xi$ be the subset of $\partial \bbB^4$ fixed by $t_\phi$. If $t_\phi$ is loxodromic (resp.\ parabolic) then $\Xi$ is two distinct points $\xi_\pm$ (resp.\ a single point $\xi_\infty$). Since $\PU(4,1)$ has real rank one, commuting loxodromic or parabolic elements necessarily have the same fixed-point set on $\partial \bbB^4$.

However, for each $\al \in \pi_1(\calC)$, the element 
\[
\rho(\phi(\al \tau_r \al^{-1})) = \rho(\phi(\al)) t_\phi \rho(\phi(\al))^{-1}
\]
is an isometry of the same type as $t_\phi$. In particular, if $\bbB^3_s$ is a hyperplane meeting $\bbB^3_r$ for some other $s \in \SR$, then $\tau_r$ commutes with $\tau_s$. Thus $\rho(\phi(\tau_s))$ commutes with $t_\phi$ and hence also has fixed set $\Xi$. \Cref{prop:ConnectedSupport} then implies that $\rho(\phi(\tau_s))$ fixes $\Xi$ for all $s \in \SR$. Since $\pi_1(\calC)$ is moreover normally generated by $\tau_r$ by \Cref{prop:exact_sequence}, it follows that $\rho(\phi(\pi_1(\calC))) = \PGam_4$ fixes $\Xi$, which is absurd.

\medskip
\noindent\textbf{Step 3.} $t_\phi$ has order dividing $6$.

\medskip
\noindent
\textbf{Proof.} Step 1 implies that the reduction of $t_\phi$ modulo $(1 - \om)$ is conjugate in $\mathrm{W}(\mathrm{E}_6)$ to the reduction of $\rho(\tau_r)$. Since the image in $\U(4,1)$ of $\tau_r$ under the monodromy $\wh{\rho}$ is a complex reflection of order six, the reduction modulo $(1 - \om)$ of the finite order element $\wh{\rho}(\phi(\tau_r))$ of $\U(4,1)$ has characteristic polynomial congruent modulo $(1 - \omega)$ to $(x+1)(x - 1)^4$. Since $t_\phi$ is elliptic by Step 2, its characteristic polynomial is a product of cyclotomic polynomials. Considering the cyclotomic polynomials of degree at most five over $\calE$ and their reductions modulo $1 - \om$, the element $t_\phi$ must have order dividing eighteen. However, every element of the kernel of reduction module $1 - \om$ has order at most three by \cite[Lem.\ 7.27]{ACT} and $\rho(\tau_r)$ has order two, implying that $\wh{\rho}(\phi(\tau_r))$ has order divisible by six. This completes Step 3.

\medskip
\noindent
\textbf{Finishing the proof of \Cref{thmKchar}.} First note that Step 3 implies that $\rho(\phi(\tau_r^6))$ is trivial in $\PGam_4$. Thus $\phi(\tau_r^6) \in K$. Since $K$ is normally generated by $\tau_r^6$ by \Cref{prop:exact_sequence}, $\phi(K) \le K$ as desired. This completes the proof of \Cref{thmKchar} \qed

\section{Proof of \Cref{theorem:main1}}\label{sec:MainProof}

Let $f \in \Aut(\calC)$ be given. Then $f$ induces an outer automorphism of $\pi_1(\calC)$. Let $\phi \in \Aut(\pi_1(\calC))$ be a representative of this outer automorphism. Recall that $K$ is the kernel of the projection $\pi_1(\calC)\to\PGam_4$. \Cref{thmKchar} gives that $\phi(K) = K$. It follows that $f$ lifts to a biholomorphic automorphism $\wh{f}$ of the cover ${\bbB^4 \ssm \calH}$ associated with $K$. Note that $\wh{f}$ is a bounded holomorphic map $(\bbB^4 \ssm \calH)\to \bbC^4$ since its image lies in $\bbB^4\subset\bbC^4$. Since $\calH$ is a locally finite union of hyperplanes, it is a nowhere dense analytic subset of $\bbB^4$. Applying the Riemann Extension Theorem \cite[Prop.\ 4.2]{Narasimhan} to each coordinate function on $\bbC^4$ gives that $\wh{f}$ extends to a holomorphic map $\wh{f} : \bbB^4 \to \bbC^4$. Then $\wh{f}(\bbB^4) \subseteq \bbB^4$ by continuity and the Open Mapping Theorem \cite[Prop.\ 1.4]{Narasimhan}.

Since $\wh{f}$ is a homotopy lift and $\PGam_4$ is the orbifold deck group of the covering $(\bbB^4 \ssm \calH)\to\calC$, it follows that for each $\gam \in \PGam_4$ there is a $\beta \in \PGam_4$ so that $\wh{f} \circ \gam = \beta \circ \wh{f}$. The Riemann extensions of $\wh{f} \circ \gam$ and $\beta \circ \wh{f}$ from $\bbB^4 \ssm \calH$ to $\bbB^4$ then agree on a nonempty open subset of $\bbB^4$, so they agree everywhere. In particular, $\wh{f}$ normalizes the extension of $\PGam_4$ from $\bbB^4\ssm\calH$ to $\bbB^4$. 

Applying the same argument to $g = f^{-1}$, lifting to an automorphism $\wh{g}$ of $\bbB^4 \ssm \calH$ with respect to the same base point, the extension of $\wh{g}$ to $\bbB^4$ has the property that $\wh{g} \circ \wh{f}$ is the identity on a nonempty open subset of $\bbB^4$. Thus $\wh{g} \circ \wh{f}$ is the identity everywhere, and so $\wh{f} \in \Aut(\bbB^4) = \PU(4,1)$. In particular, $\wh{f}$ defines an element of the normalizer of $\PGam_4$ in $\PU(4,1)$, and thus it descends to an automorphism of $X$ extending $f$. The theorem then follows from \Cref{cor:ArithNoAuts}. \qed

\begin{rem}\label{rmk1}
We spell out the alternate approach mentioned in \Cref{rmk0}, based on the Eckardt divisor. Allcock--Carlson--Toledo show that a projective biflection of a smooth cubic surface is carried under the period map to a biflection in a long root \cite[Lem.\ 11.4]{ACT}. Conversely, using Segre's theorem, they recall that any smooth cubic surface with a nontrivial projective automorphism admits a biflection \cite[(11.3), (11.12)]{ACT}. Since the monodromy group is transitive on the long roots \cite[Thm.\ 11.13]{ACT}, the long root mirror arrangement maps to a single irreducible divisor, namely the Eckardt divisor.

The generic stabilizer along this divisor is $\mathbb Z/2\mathbb Z$, generated by the Eckardt involution. Since every nontrivial automorphism of a smooth cubic surface gives a biflection, and since the long roots form a single monodromy orbit, the Eckardt divisor is the unique divisor whose generic stabilizer is $\bbZ/2\bbZ$. Any automorphism of the orbifold must therefore preserve this divisor. One can then replace our argument using the short root (i.e., or nodal) divisor by an analogous argument using the long roots (i.e., Eckardt) divisor, obtaining the same conclusion about the automorphism group.

This argument has two drawbacks. First, it does not recover the group-theoretic statement that $K$ is characteristic (\Cref{thmKchar}). Second, it relies on the special role of the Eckardt divisor as a distinguished divisorial component of the stabilizer locus, and is therefore less obviously adaptable to other moduli spaces.
\end{rem}

\bibliography{autc2}{}
\bibliographystyle{abbrv}

\end{document}